\documentclass{article}

\usepackage{amsmath}
\usepackage{amssymb}
\usepackage{amsthm}
\usepackage{enumerate}

\newtheorem{theorem}{Theorem}[section]
\newtheorem{corollary}[theorem]{Corollary}
\newtheorem{lemma}[theorem]{Lemma}
\newtheorem{proposition}[theorem]{Proposition}
\theoremstyle{definition}

\newtheorem{notation}[theorem]{Notation}
\theoremstyle{remark}
\newtheorem{remark}[theorem]{Remark}

\DeclareMathOperator{\Sq}{Sq}
\DeclareMathOperator{\rank}{rank}

\begin{document}

\title{On $(n-2)$-connected $2n$-dimensional Poincar\'e complexes with torsion-free homology}
\author{Xueqi Wang}
\date{}
\maketitle

\begin{abstract}

Let $X$ be an $(n-2)$-connected $2n$-dimensional Poincar\'e complex with torsion-free homology, where $n\geq 4$. We prove that $X$ can be decomposed into a connected sum of two Poincar\'e complexes: one being $(n-1)$-connected, while the other having trivial $n$th homology group. Under the additional assumption that $H_n(X)=0$ and $\Sq^2:H^{n-1}(X;\mathbb{Z}_2)\to H^{n+1}(X;\mathbb{Z}_2)$ is trivial, we can prove that $X$ can be further decomposed into connected sums of Poincar\'e complexes whose $(n-1)$th homology is isomorphic to $\mathbb{Z}$. As an application of this result, we classify the homotopy types of such $2$-connected $8$-dimensional Poincar\'e complexes.


\end{abstract}

\section{Introduction}

A Poincar\'e complex, roughly speaking, is a CW complex which satisfies Poincar\'e duality for arbitrary local coefficients. It can be regarded as a generalization of a manifold and reflects the homotopy properties of manifolds. It plays a crucial role in the classification of manifolds when using the Browder-Novikov-Sullivan-Wall surgery theory. For a precise definition and survey for Poincar\'e complexes, we recommend referencing Klein's work \cite{Kle00}.

A significant challenge in this field is the classification of Poincar\'e complexes satisfying specific conditions. To date, most classifications have been focused on either low-dimensional or highly connected Poincar\'e complexes. In the context of highly connected cases, Wall \cite{Wal62} has classified the homotopy types of $(n-1)$-connected $2n$-dimensional Poincar\'e complexes (see \cite[Section 4]{Kle00} for a detailed discussion). Consequently, this paper primarily focuses on $(n-2)$-connected $2n$-dimensional Poincar\'e complexes. Yamaguchi \cite{Yam90} explored the case where the $(n+1)$-skeleton of $X$ is wedges of $\Sigma^{n-3}\mathbb{C}P^2$. Sasao and Takahashi \cite{ST79} studied the case where $X$ is an $(n-2)$-connected $2n$-dimensional rational homology sphere. However, our understanding of these complexes remains limited in other cases.

In instances where a Poincar\'e complex is simply connected, its definition becomes more straightforward. We define a CW complex $X$ as a simply connected Poincar\'e complex of dimension $m$, if there exists a homology class $[X]\in H_m(X;\mathbb{Z})$ such that the cap product with $[X]$ induces isomorphisms
\[
-\cap [X]: H^q(X;\mathbb{Z})\xrightarrow{\cong} H_{m-q}(X;\mathbb{Z})
\]
for every integer $q$.

Let $X$ denote an $(n-2)$-connected $2n$-dimensional Poincar\'e complex with torsion-free homology, where $n\geq 4$. Our first main result is that $X$ can be decomposed into a connected sum of two Poincar\'e complexes: one being $(n-1)$-connected, while the other having trivial $n$th homology group.

\begin{theorem}\label{thm:1}
  Let $X$ be an $(n-2)$-connected $2n$-dimensional Poincar\'e complex with torsion-free homology ($n\geq 4$). Then $X\simeq X_1\# X_2$, where $H_n(X_1)=0$ and $X_2$ is $(n-1)$-connected.
\end{theorem}

Wall \cite{Wal62} has provided complete invariants for $(n-1)$-connected $2n$-dimensional Poincar\'e complexes. To fully understand $(n-2)$-connected $2n$-dimensional Poincar\'e complexes, we can primarily focus on the case where the $n$th homology group is trivial, as indicated by Theorem \ref{thm:1}. We prove that such a complex $X$ can be further decomposed into connected sums of the ``simplest'' ones, provided that $X$ satisfies the following hypothesis (H):

\begin{description}
  \item[Hypothesis (H)] (1) $H_n(X)=0$;

  (2) $\Sq^2:H^{n-1}(X;\mathbb{Z}_2)\to H^{n+1}(X;\mathbb{Z}_2)$ is trivial.
\end{description}

\begin{theorem}\label{thm:7}
  Let $X$ be an $(n-2)$-connected $2n$-dimensional Poincar\'e complex ($n\geq 4$) which has torsion-free homology and satisfies hypothesis (H). Then $X\simeq X_1\# X_2\# \cdots \# X_k$, where $H_{n-1}(X_i)\cong\mathbb{Z}$, $1\leq i\leq k$.
\end{theorem}

\begin{remark}
  If $X$ is an $(n-2)$-connected $2n$-dimensional Poincar\'e complex with torsion-free homology and trivial $n$th homology, then the $(n+1)$-skeleton of $X$ must be wedges of copies of $S^{n-1}\vee S^{n+1}$ and $\Sigma^{n-3}\mathbb{C}P^2$ (see Proposition \ref{thm:8}). We further assume that $\Sq^2:H^{n-1}(X;\mathbb{Z}_2)\to H^{n+1}(X;\mathbb{Z}_2)$ is trivial to rule out the $\Sigma^{n-3}\mathbb{C}P^2$ summands, as their presence would significantly complicate the calculations. In comparison to Ishimoto's results \cite{Ish73} on the classification of $(n-2)$-connected $2n$-manifolds with torsion-free homology, we may expect that $X\simeq X_1\# X_2\# \cdots \# X_k$ where $\rank H_n(X_i)\leq 2$ for $1\leq i\leq k$, provided hypothesis (H) is not satisfied.
\end{remark}

Using Theorem \ref{thm:7}, we are able to determine all possible homotopy types of $X$ for a specific $n$. For example, we can obtain a classification of $2$-connected $8$-dimensional Poincar\'e complex which has torsion-free homology and satisfies hypothesis (H). To formulate our result, we use the following notations. Let
\[
X_{r,s}=(S^3\vee S^5)\cup_{[\iota_3,\iota_5]+r\nu'\eta_6+s\eta_5^2} D^8,
\]
where $\iota_k$ denotes the identity map of $S^k$, while $\nu'\eta_6$ and $\eta_5^2$ are certain elements in $\pi_7(S^3)$ and $\pi_7(S^5)$ respectively (cf Lemma \ref{thm:16}). The composition with the inclusion into $S^3\vee S^5$ is omitted for brevity, and we use the same letter for a map and its homotopy class if there is no confusion. It is obvious that $X_{0,0}\simeq S^3\times S^5$.

\begin{theorem}\label{thm:15}
  Let $X$ be a $2$-connected $8$-dimensional Poincar\'e complex which has torsion-free homology and satisfies hypothesis (H) (where $n=4$). Suppose $\rank H_3(X)=k$, then $X$ is homotopy equivalent to one of the following complexes:

  $(1)\ \ \#_k S^3\times S^5$, \quad $(2) \ \ (\#_{k-1}S^3\times S^5)\#X_{1,0}$, \quad $(3)\ \ (\#_{k-1}S^3\times S^5)\#X_{0,1}$,

  $(4)\ \ (\#_{k-1}S^3\times S^5)\#X_{1,1}$, \quad $(5)\ \ (\#_{k-2}S^3\times S^5)\#X_{1,0}\#X_{0,1}$ (when $k\geq 2$).

  \noindent Furthermore, these complexes are not homotopy equivalent to each other.
\end{theorem}

This paper is organized as follows: Section \ref{sec:1} presents fundamental results in homotopy theory that will be used subsequently. The proofs of Theorems \ref{thm:1}, \ref{thm:7}, and \ref{thm:15} are then provided in Sections \ref{sec:2}, \ref{sec:4}, and \ref{sec:5} respectively.

\section{Preliminaries}\label{sec:1}

In this paper, we use the same letter to denote a map and its homotopy class, and we often omit the inclusion map to a lager space, if there is no confusion.

\subsection{Results on homotopy groups}

We briefly review some fundamental results of the homotopy groups of spheres. It is well known that $\pi_k(S^n)=0$ for $k<n$, and $\pi_n (S^n)\cong\mathbb{Z}$ which is generated by the identity map $\iota_n$ of $S^n$. Additionally, we present other results that may be used in subsequent discussions.

\begin{lemma}[\cite{Tod52,Tod62}]\label{thm:16}
  \begin{enumerate}
    \item $\pi_{n+1}(S^n)\cong
    \begin{cases}
      \mathbb{Z}\{\eta_2\}, &  n=2 \\
      \mathbb{Z}_2\{\eta_n\}, &  n\geq 3.
    \end{cases}$

    Here $\eta_2: S^3\to S^2$ is the Hopf map and $\eta_n=\Sigma^{n-2}\eta_2$ is the $(n-2)$-fold suspension of $\eta_2$.
    \item $\pi_{n+2}(S^n)\cong\mathbb{Z}_2\{\eta_n^2\}$ for $n\geq 2$, where $\eta_n^2=\eta_n\eta_{n+1}$.
    \item $\pi_7(S^3)\cong\mathbb{Z}_2\{\nu'\eta_6\}$, where $\nu'$ is an element of order $4$ in $\pi_6(S^3)$. Besides, one has $\eta_3\eta_4\eta_5\eta_6=0$.
  \end{enumerate}
\end{lemma}



Let $X\cup_\phi D^n$ represent the space obtained by attaching $D^n$ to $X$ via a map $\phi: S^{n-1}=\partial D^n\to X$. The homotopy type of $X\cup_\phi D^n$ only depends on the homotopy class of $\phi$. Using the cell decomposition $\mathbb{C}P^2=S^2\cup_{\eta_2}D^4$, one can easily derive the following information.

\begin{lemma}
  \begin{enumerate}
    \item $\pi_{n+2}(\Sigma^n\mathbb{C}P^2)\cong\mathbb{Z}\{\iota_{n+2}\}$;
    \item $\pi_{n+3}(\Sigma^n\mathbb{C}P^2)=0$.
  \end{enumerate}
\end{lemma}


\subsection{Additivity of composition}

\begin{lemma}[cf {\cite[Theorems X.8.2 and X.8.3]{Whi78}}]\label{thm:22}
  \begin{enumerate}
    \item If $\sigma\in\pi_n(S^r)$ is a suspended element and $\alpha_1,\alpha_2\in\pi_r(X)$, then $(\alpha_1+\alpha_2)\sigma=\alpha_1\sigma+\alpha_2\sigma$.
    \item If $\sigma\in\pi_n(S^r)$, $\alpha_1,\alpha_2\in\pi_r(X)$, and $X$ is an $H$-space, then $(\alpha_1+\alpha_2)\sigma=\alpha_1\sigma+\alpha_2\sigma$.
  \end{enumerate}
\end{lemma}

In particular, the Freudenthal suspension theorem guarantees that all elements of $\pi_{2n-1}(S^{n+1})$ are suspended. Consequently, we can derive the following corollary.

\begin{corollary}
  If $\sigma\in\pi_{2n-1}(S^{n+1})$ and $\alpha_1,\alpha_2\in\pi_{n+1}(X)$, then $(\alpha_1+\alpha_2)\sigma=\alpha_1\sigma+\alpha_2\sigma$.
\end{corollary}

Hilton \cite{Hil55} obtained a formula when $\sigma$ need not be suspended. We only present the specific cases used in our paper.

\begin{lemma}\label{thm:5}
  If $n\geq 4$, $\alpha_1,\alpha_2\in\pi_n(X)$ and $\sigma\in \pi_{2n-1}(S^n)$, then
  \[
  (\alpha_1+\alpha_2)\sigma=\alpha_1\sigma+\alpha_2\sigma+H(\sigma)[\alpha_1,\alpha_2],
  \]
  where $H:\pi_{2n-1}(S^n)\to \mathbb{Z}$ is the Hopf invariant.
\end{lemma}

\begin{lemma}\label{thm:18}
  If $\alpha_1,\alpha_2\in\pi_3(X)$, then
  \[
  (\alpha_1+\alpha_2)\nu'\eta_6=\alpha_1\nu'\eta_6+\alpha_2\nu'\eta_6+[\alpha_1,\alpha_2]\eta_5^2.
  \]
\end{lemma}

\begin{proof}
  It follows immediately from \cite[(6.1)]{Hil55} that
  \begin{multline*}
    (\alpha_1+\alpha_2)\nu'\eta_6=\alpha_1\nu'\eta_6+\alpha_2\nu'\eta_6+[\alpha_1,\alpha_2]\circ H_0(\nu'\eta_6) \\
    +[\alpha_1,[\alpha_1,\alpha_2]]\circ H_1(\nu'\eta_6)+[\alpha_2,[\alpha_1,\alpha_2]]\circ H_2(\nu'\eta_6),
  \end{multline*}
  where $H_0:\pi_7(S^3)\to \pi_7(S^5)$ and $H_1, H_2: \pi_7(S^3)\to \pi_7(S^7)$ are homomorphisms. Since $\pi_7(S^3)\cong \mathbb{Z}_2$ and $\pi_7(S^7)\cong \mathbb{Z}$, it is necessary to have $H_1,H_2=0$. The map $H_0$ is called the Hilton-Hopf invariant and coincide with the James-Hopf invariant $H$. As $H(\nu'\eta_6)=\eta_5^2$ (cf \cite[(5.7)]{Tod62}), the proof is thereby complete.
\end{proof}

By repeatedly using Lemma \ref{thm:18}, we have:

\begin{corollary}\label{thm:23}
  If $\alpha_i\in\pi_3(X)$ for $1\leq i\leq k$, then
  \[
  \bigg(\sum_{i=1}^{k}\alpha_i\bigg)\nu'\eta_6=\sum_{i=1}^{k}\alpha_i\nu'\eta_6+\sum_{1\leq i<j\leq k}[\alpha_i,\alpha_j]\eta_5^2.
  \]
\end{corollary}

\subsection{Self-maps and self-homotopy equivalences}

For topological spaces $X$ and $Y$, we denote by $[X,Y]$ the set of homotopy classes of maps from $X$ to $Y$. It is known that if $Y$ is simply connected, the set $[X,Y]$ remains unaltered if we add the assumption that the maps and homotopies are basepoint preserving. Let $\mathcal{E} (X)$ be the subset of $[X,X]$ that consists of all self-homotopy equivalences of $X$. Notably, $\mathcal{E} (X)$ forms a group under the composition of self-maps; however,this group structure is not required in this paper. The following lemma is well-known (cf \cite[Lemma 4.10]{Yam90}), and is fundamental to our subsequent discussion on homotopy types.

\begin{lemma}\label{thm:12}
  Let $X$ be a simply connected $m$-dimensional CW-complex and $n\geq m+2$. Then $X\cup_\phi D^n\simeq X\cup_\psi D^n$ if and only if there exists $f\in\mathcal{E}(X)$ such that $\psi=\pm f\phi\in \pi_{n-1}(X)$.
\end{lemma}

\section{Proof of Theorem \ref{thm:1}}\label{sec:2}

Let $X$ be an $(n-2)$-connected $2n$-dimensional Poincar\'e complex with torsion-free homology ($n\geq 4$). Assuming that $\rank H_{n-1}(X)=k$ and $\rank H_n(X)=l$, it follows from Poincar\'e duality that the integral homology groups of $X$ are
\[
H_i(X)\cong
\begin{cases}
  \mathbb{Z}, & i=0,2n \\
  \mathbb{Z}^k, & i=n-1,n+1 \\
  \mathbb{Z}^l, &  i=n \\
  0, & \mbox{otherwise}.
\end{cases}
\]
Consequently, $X$ possesses a ``minimal cell structure" up to homotopy equivalence (cf \cite[Proposition 4C.1]{Hat02})
\[
X\simeq \bigg(\bigvee_{i=1}^k S^{n-1}_i\bigg)\bigcup \bigg(\coprod_{s=1}^{l}D^n_s\bigg)\bigcup \bigg(\coprod_{j=1}^{k}D^{n+1}_j\bigg)\bigcup D^{2n}.
\]
We use symbols $S^p_i$ and $D^q_j$ to denote $p$-spheres and $q$-disks labeled $i$ and $j$ respectively.
Let $X^{(n+1)}$ denote the $(n+1)$-skeleton of the space on the right hand side.

\begin{proposition}\label{thm:8}
  $X^{(n+1)}\simeq \big(\bigvee_{i=1}^r (S^{n-1}_i\vee S^{n+1}_i)\big) \vee\big(\bigvee_{i=r+1}^k \Sigma^{n-3}\mathbb{C}P^2_i\big)\vee\big(\bigvee_{j=1}^l S^n_j\big)$ for some integer $1\leq r\leq k$.
\end{proposition}

\begin{proof}
  This lemma is a direct consequence of the classification of $\mathbf{A}_n^2$-complexes by Chang \cite{Cha50}. Recall that $\mathbf{A}_n^2$ represents the homotopy category which includes $(n-1)$-connected finite CW-complexes with dimension at most $n+2$. Note that $X^{(n+1)}$ belongs to $\mathbf{A}_{n-1}^2$. According to Chang's result, complexes in $\mathbf{A}_n^2$ ($n\geq 3$) with torsion-free homology can be decomposed into wedge sums of $S^n$, $S^{n+1}$, $S^{n+2}$, and $\Sigma^{n-2}\mathbb{C}P^2$. Therefore, together with the homology constraints $H_{n-1}(X^{(n+1)})\cong H_{n+1}(X^{(n+1)})\cong\mathbb{Z}^k$ and $H_n(X^{(n+1)})\cong \mathbb{Z}^l$, we can derive our proposition.
\end{proof}

Let
\[
K_1=\big(\bigvee_{i=1}^r (S^{n-1}_i\vee S^{n+1}_i)\big) \vee\big(\bigvee_{i=r+1}^k \Sigma^{n-3}\mathbb{C}P^2_i\big), \quad K_2=\bigvee_{j=1}^l S^n_j.
\]
From now on, we identify $X^{(n+1)}$ with $K_1\vee K_2$ and assume that
\[
X=(K_1\vee K_2)\cup_\phi D^{2n}.
\]
Denote by $\iota_n^i$ the inclusion map that identifies $S^n$ with $S^n_i$. Lemmas \ref{thm:24} and \ref{thm:2} are direct consequences of the Hilton-Milnor theorem \cite{Hil55,Mil72} (see \cite[Section 4]{Bar60} for the version we need).

\begin{lemma}\label{thm:24}
  $\pi_{2n-1}(K_2)\cong (\bigoplus_{i=1}^l\pi_{2n-1}(S^n_i))\oplus(\bigoplus_{1\leq i<j\leq l}\mathbb{Z}\{[\iota_n^i,\iota_n^j]\})$.
\end{lemma}

\begin{lemma}\label{thm:2}
    $\pi_{2n-1}(X^{(n+1)})\cong\pi_{2n-1}(K_1)\oplus\pi_{2n-1}(K_2)\oplus G$, where
\[
G\cong \bigoplus_{\substack{1\leq i\leq r\\1\leq j\leq l}}\mathbb{Z}_2\{[\iota_{n-1}^i,\iota_n^j]\eta_{2n-2}\}.
\]
\end{lemma}

For elements in $G$, the following holds.

\begin{lemma}\label{thm:4}
  $[\iota_{n-1}^i\eta_{n-1},\iota_n^j]=[\iota_n^j,\iota_{n-1}^i\eta_{n-1}]=[\iota_{n-1}^i,\iota_n^j]\eta_{2n-2}$.
\end{lemma}

\begin{proof}
  Since $\eta_n$ is of order $2$ for $n\geq 3$, the Whitehead product is commutative in our situation, ie,
  \[
  [\iota_{n-1}^i\eta_{n-1},\iota_n^j]=[\iota_n^j,\iota_{n-1}^i\eta_{n-1}].
  \]
  Besides, we have
  \[
  [\iota_{n-1}^i,\iota_n^j]\eta_{2n-2}=[\iota_{n-1}^i,\iota_n^j]\Sigma(\eta_{n-2}\wedge\iota_{n-1}) =[\iota_{n-1}^i\eta_{n-1},\iota_n^j].
  \]
\end{proof}

Under the decomposition of Lemma \ref{thm:2}, suppose
\[
\phi=\phi_1+\phi_2+\phi_3\in \pi_{2n-1}(X^{(n+1)}),
\]
where
\[
\phi_1\in \pi_{2n-1}(K_1),\quad \phi_2\in\pi_{2n-1}(K_2), \quad \phi_3=\sum_{\substack{1\leq i\leq r\\1\leq j\leq l}}a_{ij}[\iota_{n-1}^i,\iota_n^j]\eta_{2n-2}\in G.
\]
Set $A_\phi:=(a_{ij})_{r\times l}$, and note that $a_{ij}\in \mathbb{Z}_2$. By definition,
\[
A_\phi=0\Leftrightarrow \phi_3=0 \Leftrightarrow \phi\in\pi_{2n-1}(K_1)\oplus \pi_{2n-1}(K_2).
\]
In this case, one can see that $X$ admits a connected sum decomposition as in Theorem \ref{thm:1}. More generally, by Lemma \ref{thm:12}, the same conclusion holds if we can find some $f\in \mathcal{E}(X^{(n+1)})$ such that $A_{f\phi}=0$.

Let $f:X^{(n+1)}\to X^{(n+1)}$ be the self-map which satisfies that
\[
f|_{S^n_j}=\iota_n^j+\sum_{s=1}^{r}c_{sj}\iota_{n-1}^s\eta_{n-1}, \quad 1\leq j\leq l,
\]
where $c_{sj}\in\mathbb{Z}_2$, and $f|_{K_1}$ is an inclusion. It is clear that $f\in\mathcal{E}(X^{(n+1)})$. Set $C_f:=(c_{ij})_{r\times l}$.

Before analyzing $A_{f\phi}$, by applying Lemma \ref{thm:24}, we further assume that
\[
\phi_2=\sum_{i=1}^{l}\iota_n^i\beta_i+\sum_{1\leq i<j\leq l}b_{ij}[\iota_n^i,\iota_n^j],
\]
where
$\beta_i\in \pi_{2n-1}(S^n)$ for $1\leq i\leq l$. We define
\[
b_{ij}=
\begin{cases}
  (-1)^n b_{ji}, & \mbox{if } i>j \\
  H(\beta_i), & \mbox{if } i=j,
\end{cases}
\]
where $H:\pi_{2n-1}(S^n)\to \mathbb{Z}$ is the Hopf invariant. Then $B_\phi:=(b_{ij})_{l\times l}$ is an invertible integral matrix and determines the intersection form of $X$ (cf \cite[page 182]{Wal62}).

\begin{lemma}\label{thm:6}
  $A_{f\phi}=A_\phi+C_fB_\phi$.
\end{lemma}

\begin{proof}
  Since $f|_{K_1}$ is an inclusion, one sees that $f\phi_1=\phi_1\in \pi_{2n-1}(K_1)$.

  For $f\phi_3$, we have
  \begin{align*}
    f\phi_3 & =\sum_{\substack{1\leq i\leq r\\1\leq j\leq l}}a_{ij}\bigg[\iota_{n-1}^i,\iota_n^j+\sum_{s=1}^{r}c_{sj}\iota_{n-1}^s\eta_{n-1}\bigg]\eta_{2n-2}\\
     & =\sum_{\substack{1\leq i\leq r\\1\leq j\leq l}}a_{ij}[\iota_{n-1}^i,\iota_n^j]\eta_{2n-2}+\sum_{\substack{1\leq i,s\leq r\\1\leq j\leq l}}a_{ij}c_{sj}[\iota_{n-1}^i,\iota_{n-1}^s\eta_{n-1}]\eta_{2n-2} \\
     & =\sum_{\substack{1\leq i\leq r\\1\leq j\leq l}}a_{ij}[\iota_{n-1}^i,\iota_n^j]\eta_{2n-2}+\psi_1, \text{\ where\ }\psi_1\in \pi_{2n-1}(K_1).
  \end{align*}

  Then we compute $f\phi_2$. Using Lemmas \ref{thm:4} and \ref{thm:5}, we can derive that
  \begin{align*}
    f(\iota_n^i\beta_i) & = \bigg(\iota_n^i+\sum_{s=1}^{r}c_{si}\iota_{n-1}^s\eta_{n-1}\bigg)\beta_i =\iota_n^i\beta_i+\sum_{s=1}^{r}c_{si}\iota_{n-1}^s\eta_{n-1}\beta_i \\
     & +H(\beta_i)\sum_{s=1}^{r}c_{si}[\iota_n^i, \iota_{n-1}^s\eta_{n-1}]+H(\beta_i)\sum_{1\leq s<t\leq r}c_{si}c_{ti}[\iota_{n-1}^s\eta_{n-1},\iota_{n-1}^t\eta_{n-1}] \\
     & = b_{ii}\sum_{s=1}^{r}c_{si}[\iota_{n-1}^s,\iota_n^i]\eta_{2n-2}+\psi_2,
  \end{align*}
  where $\psi_2\in \pi_{2n-1}(K_1)\oplus \pi_{2n-1}(K_2)$. For $[\iota_n^i,\iota_n^j]$,
  \begin{align*}
    f([\iota_n^i,\iota_n^j]) & = \bigg[\iota_n^i+\sum_{s=1}^{r}c_{si}\iota_{n-1}^s\eta_{n-1}, \iota_n^j+\sum_{s=1}^{r}c_{sj}\iota_{n-1}^s\eta_{n-1}\bigg] \\
     & =[\iota_n^i,\iota_n^j]+\sum_{s=1}^{r}c_{sj}[\iota_n^i,\iota_{n-1}^s\eta_{n-1}]+ \sum_{s=1}^{r}c_{si}[\iota_{n-1}^s\eta_{n-1},\iota_n^j] \\
      & +\sum_{1\leq s,t\leq r}c_{si}c_{tj}[\iota_{n-1}^s\eta_{n-1},\iota_{n-1}^t\eta_{n-1}] \\
      & =\sum_{s=1}^{r}c_{sj}[\iota_{n-1}^s,\iota_n^i]\eta_{2n-2}+ \sum_{s=1}^{r}c_{si}[\iota_{n-1}^s,\iota_n^j]\eta_{2n-2}+\psi_3,
  \end{align*}
  where $\psi_3\in \pi_{2n-1}(K_1)\oplus \pi_{2n-1}(K_2)$. Combining these, we obtain that
  \begin{multline*}
    f\phi_2= \sum_{\substack{1\leq i\leq l\\1\leq s\leq r}}b_{ii}c_{si}[\iota_{n-1}^s,\iota_n^i]\eta_{2n-2}+\sum_{\substack{1\leq i<j\leq l\\1\leq s\leq r}}b_{ij}c_{sj}[\iota_{n-1}^s,\iota_n^i]\eta_{2n-2} \\
    + \sum_{\substack{1\leq i<j\leq l\\1\leq s\leq r}}b_{ij}c_{si}[\iota_{n-1}^s,\iota_n^j]\eta_{2n-2}+\psi,
  \end{multline*}
  where $\psi\in\pi_{2n-1}(K_1)\oplus \pi_{2n-1}(K_2)$. Given that $b_{ij}=(-1)^nb_{ji}$ and $[\iota_{n-1}^i,\iota_n^j]\eta_{2n-2}$ is of order $2$, the second term can be modified as follows:
  \begin{align*}
    \sum_{\substack{1\leq i<j\leq l\\1\leq s\leq r}}b_{ij}c_{sj}[\iota_{n-1}^s,\iota_n^i]\eta_{2n-2} & =\sum_{\substack{1\leq j<i\leq l\\1\leq s\leq r}}b_{ji}c_{si}[\iota_{n-1}^s,\iota_n^j]\eta_{2n-2} \\
     & =\sum_{\substack{1\leq j<i\leq l\\1\leq s\leq r}}b_{ij}c_{si}[\iota_{n-1}^s,\iota_n^j]\eta_{2n-2}.
  \end{align*}
  Consequently,
  \begin{align*}
     f\phi_2 &  = \sum_{\substack{1\leq i,j\leq l\\1\leq s\leq r}}b_{ij}c_{si}[\iota_{n-1}^s,\iota_n^j]\eta_{2n-2}+\psi= \sum_{\substack{1\leq i\leq r\\1\leq j\leq l}}\bigg(\sum_{s=1}^{l}c_{is}b_{sj}\bigg)[\iota_{n-1}^i,\iota_n^j]\eta_{2n-2}+\psi.
  \end{align*}

  Combining the above calculations,
  \[
  f\phi= \sum_{\substack{1\leq i\leq r\\1\leq j\leq l}}\bigg(a_{ij}+\sum_{s=1}^{l}c_{is}b_{sj}\bigg)[\iota_{n-1}^i,\iota_n^j]\eta_{2n-2}+\psi',
  \]
  where $\psi'\in\pi_{2n-1}(K_1)\oplus \pi_{2n-1}(K_2)$. Thus $A_{f\phi}$ is an $r\times l$ matrix whose entry in the $i$th row and $j$th column is $a_{ij}+\sum_{s=1}^{l}c_{is}b_{sj}$. The lemma follows directly.
\end{proof}

\begin{proof}[Proof of Theorem \ref{thm:1}]
  Let $f$ be the self-map of $K_1\vee K_2$ which satisfies that
\[
f|_{S^n_j}=\iota_n^j+\sum_{s=1}^{r}c_{sj}\iota_{n-1}^s\eta_{n-1}, \quad 1\leq j\leq l,
\]
where the matrix $C_f=(c_{ij})_{r\times l}$ is defined to be $A_{\phi}B_{\phi}^{-1}$, and $f|_{K_1}$ is the inclusion. It is evident that $f\in \mathcal{E}(K_1\vee K_2)$. Since entries of $A_\phi$ are from $\mathbb{Z}_2$, we can apply Lemma \ref{thm:6} to conclude that
\[
A_{f\phi}=A_\phi+C_fB_\phi=A_\phi+A_{\phi}B_{\phi}^{-1}B_{\phi}=2A_\phi=0.
\]
Therefore, $f\phi$ can be decomposed into $\phi_1+\phi_2$, where $\phi_1\in\pi_{2n-1}(K_1)$ and $\phi_2\in\pi_{2n-1}(K_2)$. Utilizing Lemma \ref{thm:12}, we infer that
\[
X\simeq (K_1\vee K_2)\cup_{\phi} D^{2n}\simeq (K_1\vee K_2)\cup_{f\phi} D^{2n}\simeq (K_1\cup_{\phi_1}D^{2n})\# (K_2\cup_{\phi_2}D^{2n}).
\]
Obviously, $H_n(K_1\cup_{\phi_1}D^{2n})=0$ and $K_2\cup_{\phi_2}D^{2n}$ is $(n-1)$-connected. Thus the proof is now fully established.
\end{proof}

\section{Proof of Theorem \ref{thm:7}}\label{sec:4}

In this section, let $X$ be an $(n-2)$-connected $2n$-dimensional Poincar\'e complex ($n\geq 4$) which has torsion-free homology and satisfies hypothesis (H). Assuming that $\rank H_{n-1}(X)=k$, Proposition \ref{thm:8} indicates that $X$ has a minimal cell structure where the $(n+1)$-skeleton $X^{(n+1)}\simeq \bigvee_{i=1}^k (S^{n-1}_i\vee S^{n+1}_i)$. To prove Theorem \ref{thm:7}, further details regarding the action of $\mathcal{E}(\bigvee_{i=1}^k (S^{n-1}_i\vee S^{n+1}_i))$ on $\pi_{2n-1}(\bigvee_{i=1}^k (S^{n-1}_i\vee S^{n+1}_i))$ are required.

We first need to make some preparations on the sets $[\bigvee_{i=1}^k S^n_i,\bigvee_{i=1}^k S^n_i]$ and $\mathcal{E}(\bigvee_{i=1}^k S^n_i)$. Let $i_r: S^n_r\to \bigvee_{i=1}^k S^n_i$ be the inclusion and $p_r:\bigvee_{i=1}^k S^n_i\to S^n_r$ be the projection ($1\leq r\leq k$). A mapping $f\in [\bigvee_{i=1}^k S^n_i,\bigvee_{i=1}^k S^n_i]$ can be associated with the matrix
\begin{equation*}
  \begin{pmatrix}
  p_1fi_1 & p_1fi_2 & \cdots & p_1fi_k \\
  p_2fi_1 & p_2fi_2 & \cdots & p_2fi_k \\
  \vdots & \vdots &  & \vdots \\
  p_kfi_1 & p_kfi_2 & \cdots & p_kfi_k
\end{pmatrix}
\end{equation*}
With this association, it is easy to verify the following.

\begin{lemma}
  \begin{enumerate}
    \item $[\bigvee_{i=1}^k S^n_i,\bigvee_{i=1}^k S^n_i]\cong M_k(\mathbb{Z})$;
    \item $\mathcal{E}(\bigvee_{i=1}^k S^n_i)\cong GL_k(\mathbb{Z})$.
  \end{enumerate}
\end{lemma}

We denote by $M_n(G)$ the set of $n\times n$ matrices whose entries are in the group $G$, and $GL_n(G)$ the subset consisting of invertible matrices.

\begin{notation}\label{n:1}
  We use $Q_f$ to represent the integral matrix corresponding to $f\in [\bigvee_{i=1}^k S^n_i,\bigvee_{i=1}^k S^n_i]$.
\end{notation}

We now consider the set $\mathcal{E}(\bigvee_{i=1}^k (S^{n-1}_i\vee S^{n+1}_i))$. For the convenience of our description, let $T_1=\bigvee_{i=1}^k S^{n-1}_i$ and $T_2=\bigvee_{i=1}^k S^{n+1}_i$, then $\bigvee_{i=1}^k (S^{n-1}_i\vee S^{n+1}_i)\cong T_1\vee T_2$.  Let
\[
i_1:T_1\to T_1\vee T_2, \qquad  i_2:T_2\to T_1\vee T_2
\]
be the inclusions and
\[
p_1:T_1\vee T_2\to T_1, \qquad p_2:T_1\vee T_2\to T_2
\]
be the projections. It is evident that $[T_1\vee T_2,T_1\vee T_2]\cong [T_1\vee T_2,T_1\times T_2]$ and $[T_1,T_2]=0$. Thus we can identify a map $F\in [T_1\vee T_2,T_1\vee T_2]$ with the matrix
$\begin{pmatrix}
p_1 F i_1 & p_1 F i_2 \\
0 & p_2 F i_2
\end{pmatrix}$. This yields the following result.

\begin{lemma}\label{thm:3}
   Elements in $\mathcal{E}(\bigvee_{i=1}^k (S^{n-1}_i\vee S^{n+1}_i))$ are in one-to-one correspondence with matrices of the form
   $\begin{pmatrix}
     f & g \\
     0 & h
   \end{pmatrix}$, where
   $
   f\in \mathcal{E}(\bigvee_{i=1}^k S^{n-1}_i)$, $g\in [\bigvee_{i=1}^k S^{n+1}_i,\bigvee_{i=1}^k S^{n-1}_i]$, and $ h\in \mathcal{E}(\bigvee_{i=1}^k S^{n+1}_i)
   $.
\end{lemma}

Then we discuss the action of $\mathcal{E}(\bigvee_{i=1}^k (S^{n-1}_i\vee S^{n+1}_i))$ on $\pi_{2n-1}(\bigvee_{i=1}^k (S^{n-1}_i\vee S^{n+1}_i))$. By the Hilton-Milnor theorem, it is direct to check that, for $n\geq 5$,
\begin{multline*}
  \pi_{2n-1}(\bigvee_{i=1}^k (S^{n-1}_i\vee S^{n+1}_i)) \cong  \bigg(\bigoplus_{i=1}^k\pi_{2n-1}(S^{n-1}_i)\bigg) \oplus \bigg(\bigoplus_{i=1}^k\pi_{2n-1}(S^{n+1}_i)\bigg) \\
  \oplus \bigg(\bigoplus_{1\leq i,j\leq k}\mathbb{Z}\{[\iota_{n-1}^i,\iota_{n+1}^j]\}\bigg)
  \oplus\bigg(\bigoplus_{1\leq i<j\leq k}\mathbb{Z}_2\{[\iota_{n-1}^i,\iota_{n-1}^j]\eta_{2n-3}^2\}\bigg)
\end{multline*}
and
\begin{multline*}
   \pi_7(\bigvee_{i=1}^k (S^3_i\vee S^5_i)) \cong  \bigg(\bigoplus_{i=1}^k\pi_7(S^3_i)\bigg) \oplus \bigg(\bigoplus_{i=1}^k\pi_7(S^5_i)\bigg)
   \oplus \bigg(\bigoplus_{1\leq i,j\leq k}\mathbb{Z}\{[\iota_3^i,\iota_5^j]\}\bigg) \\
  \oplus\bigg(\bigoplus_{1\leq i<j\leq k}\mathbb{Z}_2\{[\iota_3^i,\iota_3^j]\eta_5^2\}\bigg)\oplus \bigg(\bigoplus_{\substack{1\leq r,s,t\leq k\\ s\leq r,\ \ s<t}}\mathbb{Z}\{[\iota_3^r,[\iota_3^s,\iota_3^t]]\}\bigg)
\end{multline*}
For the convenience of our discussion, let
\begin{gather*}
  B=\bigg(\bigoplus_{i=1}^k\pi_{2n-1}(S^{n-1}_i)\bigg) \oplus \bigg(\bigoplus_{i=1}^k\pi_{2n-1}(S^{n+1}_i)\bigg),\\
  C=\begin{cases}
      \bigoplus_{1\leq i<j\leq k}\mathbb{Z}_2\{[\iota_{n-1}^i,\iota_{n-1}^j]\eta_{2n-3}^2\}, & n\geq 5 \\
      \bigg(\bigoplus_{1\leq i<j\leq k}\mathbb{Z}_2\{[\iota_3^i,\iota_3^j]\eta_5^2\}\bigg)\oplus \bigg(\bigoplus_{\substack{1\leq r,s,t\leq k\\ s\leq r,\ \ s<t}}\mathbb{Z}\{[\iota_3^r,[\iota_3^s,\iota_3^t]]\}\bigg), & n=4.
    \end{cases}
\end{gather*}
Then
\[
\pi_{2n-1}(\bigvee_{i=1}^k (S^{n-1}_i\vee S^{n+1}_i))\cong \bigg(\bigoplus_{1\leq i,j\leq k}\mathbb{Z}\{[\iota_{n-1}^i,\iota_{n+1}^j]\}\bigg)\oplus B\oplus C.
\]
Assuming
$
X\simeq\bigg(\bigvee_{i=1}^k (S^{n-1}_i\vee S^{n+1}_i)\bigg)\bigcup_\phi D^{2n}
$,
it is evident that Theorem \ref{thm:7} holds if we can find some $F\in \mathcal{E}(\bigvee_{i=1}^k (S^{n-1}_i\vee S^{n+1}_i))$ such that $F\phi=\sum_{i=1}^k [\iota_{n-1}^i,\iota_{n+1}^i]+\psi$, where $\psi\in B$. This is our goal. To accomplish this, certain preparatory steps are required.

Suppose $\phi\in\pi_{2n-1}(\bigvee_{i=1}^k (S^{n-1}_i\vee S^{n+1}_i))$ is decomposed as
\[
\phi=\sum_{1\leq i,j\leq k} a_{ij}[\iota_{n-1}^i,\iota_{n+1}^j]+\phi',
\]
where $\phi'\in B\oplus C$. Set $M_\phi:=(a_{ij})_{k\times k}$. By definition, we have
\begin{equation}\label{eq:3}
  M_{\phi+\psi}=M_\phi+M_\psi
\end{equation}
and
\begin{equation}\label{eq:4}
  M_\phi=0 \iff \phi\in B\oplus C.
\end{equation}
Let $F=\begin{pmatrix}
            f & g \\
            0 & h
         \end{pmatrix}$ be an element of $\mathcal{E}(\bigvee_{i=1}^k (S^{n-1}_i\vee S^{n+1}_i))$ as per the notation in Lemma \ref{thm:3}.

\begin{lemma}\label{thm:9}
  If $M_\phi=0$, then $M_{F\phi}=0$.
\end{lemma}

\begin{proof}
  This lemma is equivalent to the statement that if $\phi\in B\oplus C$, then $F\phi$ also belongs to $B\oplus C$. We only need to verify this for the generators of $B\oplus C$. By Serre finiteness theorem \cite{Ser53}, $B$ is finite. Since $F\phi$ must be torsion if $\phi$ is, and therefore cannot contain any $[\iota_{n-1}^i,\iota_{n+1}^j]$, we only need to check for $\phi=[\iota_3^r,[\iota_3^s,\iota_3^t]]$ in the case $n=4$. Note that the image of $[\iota_3^r,[\iota_3^s,\iota_3^t]]$ lies in $\bigvee_{i=1}^k S^3_i$, and therefore so does that of $F[\iota_3^r,[\iota_3^s,\iota_3^t]]$ (after homotopy), which means $F\phi$ does not contain any $[\iota_3^i,\iota_5^j]$. Thus the lemma is proved.
\end{proof}

Recall that we denote by $Q_f$ the integral matrix corresponding to $f\in [\bigvee_{i=1}^k S^n_i,\bigvee_{i=1}^k S^n_i]$ (see Notation \ref{n:1}). In general, we have:

\begin{lemma}\label{thm:11}
  $M_{F\phi}=Q_fM_{\phi}Q_h^T$.
\end{lemma}

\begin{proof}
  Firstly, we deal with the case where $\phi=\sum_{1\leq i,j\leq k} a_{ij}[\iota_{n-1}^i,\iota_{n+1}^j]$.
  We have
  \begin{multline*}
    F\phi=\sum_{1\leq i,j\leq k} a_{ij}[F\iota_{n-1}^i,F\iota_{n+1}^j] =\sum_{1\leq i,j\leq k} a_{ij}[f\iota_{n-1}^i,g\iota_{n+1}^j+h\iota_{n+1}^j] \\
    =\sum_{1\leq i,j\leq k} a_{ij}[f\iota_{n-1}^i,g\iota_{n+1}^j]+\sum_{1\leq i,j\leq k} a_{ij}[f\iota_{n-1}^i,h\iota_{n+1}^j].
  \end{multline*}
  The image of both $f$ and $g$ lies in $\bigvee_{i=1}^k S^{n-1}_i$, implying that $[f\iota_{n-1}^i,g\iota_{n+1}^j]$ does too. Hence we only need to analyze the part $\sum_{1\leq i,j\leq k} a_{ij}[f\iota_{n-1}^i,h\iota_{n+1}^j]$. Now the equation $M_{F\phi}=Q_fM_{\phi}Q_h^T$ can be derived through straightforward linear algebra calculations.

  In general, suppose $\phi=\phi'+\psi$, where $\phi'=\sum_{1\leq i,j\leq k} a_{ij}[\iota_{n-1}^i,\iota_{n+1}^j]$ and $\psi\in B\oplus C$. Then the above case together with Lemma \ref{thm:9}, (\ref{eq:3}) and (\ref{eq:4}) imply that
  \[
  M_{F\phi}=M_{F\phi'}+M_{F\psi}=M_{F\phi'}=Q_fM_{\phi'}Q_h^T=Q_fM_{\phi}Q_h^T.
  \]
\end{proof}

\begin{lemma}\label{thm:20}
  Assuming $n\geq 5$, if $\phi\in\pi_{2n-1}(\bigvee_{i=1}^k (S^{n-1}_i\vee S^{n+1}_i))$ is of the form  $\phi=\sum_{i=1}^k [\iota_{n-1}^i,\iota_{n+1}^i]+\psi$ with $\psi\in B\oplus C$, then for any $\alpha\in C$, there exists $F\in\mathcal{E}(\bigvee_{i=1}^k (S^{n-1}_i\vee S^{n+1}_i))$ such that $F\phi-(\phi+\alpha)\in B$.
\end{lemma}

\begin{proof}
  First, we prove this lemma when $\alpha$ is one of the generators of $C$, ie, $[\iota_{n-1}^i,\iota_{n-1}^j]\eta_{2n-3}^2$.

  If $\alpha=[\iota_{n-1}^i,\iota_{n-1}^j]\eta_{2n-3}^2$, let $F_{ij}$ be the self-map of $\bigvee_{i=1}^k (S^{n-1}_i\vee S^{n+1}_i)$ that satisfies $F_{ij}|_{S^{n+1}_i}=\iota_{n+1}^i+\iota_{n-1}^j\eta_{n-1}^2$ and its restriction to other parts is an inclusion. Then $F_{ij}\in \mathcal{E}(\bigvee_{i=1}^k (S^{n-1}_i\vee S^{n+1}_i))$. We have
        \begin{align*}
          F_{ij}([\iota_{n-1}^i,\iota_{n+1}^i]) & =[\iota_{n-1}^i,\iota_{n+1}^i+\iota_{n-1}^j\eta_{n-1}^2] \\
           & =[\iota_{n-1}^i,\iota_{n+1}^i]+[\iota_{n-1}^i,\iota_{n-1}^j\eta_{n-1}^2] \\
           & =[\iota_{n-1}^i,\iota_{n+1}^i]+[\iota_{n-1}^i,\iota_{n-1}^j]\eta_{2n-3}^2,
        \end{align*}
        and for any $\beta\in\pi_{2n-1}(S^{n+1})$,
        \[
        F_{ij}(\iota_{n+1}^i\beta)=(\iota_{n+1}^i+\iota_{n-1}^j\eta_{n-1}^2)\beta= \iota_{n+1}^i\beta+\iota_{n-1}^j\eta_{n-1}^2\beta\in B.
        \]
  By definition, $F_{ij}$ fixes $[\iota_{n-1}^s,\iota_{n+1}^s]$ ($s\neq i$) and elements in $\pi_{2n-1}(S^{n-1}_i)$ and $C$. Therefore, in this case,
  \[
  F_{ij}\phi-(\phi+\alpha)=F_{ij}\phi-(\phi+[\iota_{n-1}^i,\iota_{n-1}^j]\eta_{2n-3}^2)\in B.
  \]

  It is important to note that in the above verification, for each generator $\alpha$ of $C$, we construct an associated map $F\in\mathcal{E}(\bigvee_{i=1}^k (S^{n-1}_i\vee S^{n+1}_i))$ which satisfies that
  \begin{enumerate}
    \item $F\phi-(\phi+\alpha)\in B$;
    \item $F$ fixes elements in $C$;
    \item $F(B)\subset B$.
  \end{enumerate}
  Therefore, if $\alpha$ is generally the sum of some generators of $C$, we can construct a map $F\in\mathcal{E}(\bigvee_{i=1}^k (S^{n-1}_i\vee S^{n+1}_i))$ through the composition of those corresponding to the generators used in the summation. It is straightforward to verify that such an $F$ meets our specified requirements.
\end{proof}

Lemma \ref{thm:20} can be strengthened as follows when $n=4$.

\begin{lemma}\label{thm:19}
  If $\phi=\sum_{i=1}^k [\iota_3^i,\iota_5^i]+\psi\in\pi_7(\bigvee_{i=1}^k (S^3_i\vee S^5_i))$ with $\psi\in B\oplus C$, then for any $\alpha\in C$, there exists $F\in\mathcal{E}(\bigvee_{i=1}^k (S^3_i\vee S^5_i))$ such that $F\phi=\phi+\alpha$.
\end{lemma}

\begin{proof}
  We first show this when $\alpha$ is one of the generators of $C$, ie, $[\iota_3^i,\iota_3^j]\eta_5^2$ and $[\iota_3^r,[\iota_3^s,\iota_3^t]]$.

  (1) If $\alpha=[\iota_3^i,\iota_3^j]\eta_5^2$, let $F_{ij}$ be constructed as in the proof of Lemma \ref{thm:20}. We have already verified that
  \[
  F_{ij}([\iota_3^i,\iota_5^i])=[\iota_3^i,\iota_5^i]+[\iota_3^i,\iota_3^j]\eta_5^2,
  \]
  and $F_{ij}$ fixes $[\iota_3^s,\iota_5^s]$ ($s\neq i$) as well as elements in $\pi_7(S^3_i)$ and $C$. By Lemma \ref{thm:16}, $\pi_7(S^5)$ is generated by $\eta_5^2$, and we have
  \[
  F_{ij}(\iota_5^i\eta_5^2)=(\iota_5^i+\iota_3^j\eta_3^2)\eta_5^2= \iota_5^i\eta_5^2+\iota_3^j\eta_3\eta_4\eta_5\eta_6=\iota_5^i\eta_5^2.
  \]
  Thus in this case $F_{ij}\phi=\phi+[\iota_3^i,\iota_3^j]\eta_5^2=\phi+\alpha$.

  (2) If $\alpha=[\iota_3^r,[\iota_3^s,\iota_3^t]]$, let $G_{rst}$ be the self-map of $\bigvee_{i=1}^k (S^3_i\vee S^5_i)$ that satisfies the condition $G_{rst}|_{S^5_r}=\iota_5^r+[\iota_3^s,\iota_3^t]$ and its restriction to other parts is an inclusion. Then $G_{rst}\in \mathcal{E}(\bigvee_{i=1}^k (S^3_i\vee S^5_i))$. We have
      \begin{gather*}
        G_{rst}([\iota_3^r,\iota_5^r])=[\iota_3^r,\iota_5^r+[\iota_3^s,\iota_3^t]]= [\iota_3^r,\iota_5^r]+[\iota_3^r,[\iota_3^s,\iota_3^t]], \\
        G_{rst}(\iota_5^r\eta_5^2)=(\iota_5^r+[\iota_3^s,\iota_3^t])\eta_5^2= \iota_5^r\eta_5^2+[\iota_3^s,\iota_3^t]\eta_5^2,
      \end{gather*}
      and $G_{rst}$ fixes $[\iota_3^i,\iota_5^i]$ ($i\neq r$) as well as elements in $\pi_{7}(S^{3}_i)$ and $C$, by definition. Therefore,
      \[
      G_{rst}\phi= \phi+[\iota_3^r,[\iota_3^s,\iota_3^t]]+[\iota_3^s,\iota_3^t]\eta_5^2\text{\quad or\quad } \phi+[\iota_3^r,[\iota_3^s,\iota_3^t]],
      \]
      which depends on whether or not $\psi$ contains a summand of $\iota_5^r\eta_5^2$. If the former case occurs, one can then composite $G_{rst}$ with $F_{st}$ to produce the necessary map.

      In the general case, $\alpha$ can be decomposed into sums of some generators of $C$. Consequently, a required $F$ can be constructed through a composition of some $F_{ij}$ and $G_{rst}$ defined previously.
\end{proof}

In Lemmas \ref{thm:20} and \ref{thm:19}, write $\psi=\psi_1+\psi_2$, where $\psi_1\in B$ and $\psi_2\in C$. Let $\alpha=\phi_2$, then the following corollary is immediate.

\begin{corollary}\label{thm:13}
  If $\phi=\sum_{i=1}^k [\iota_{n-1}^i,\iota_{n+1}^i]+\psi$ is an element of $\pi_{2n-1}(\bigvee_{i=1}^k (S^{n-1}_i\vee S^{n+1}_i))$ with $\psi\in B\oplus C$, then there exists $F\in\mathcal{E}(\bigvee_{i=1}^k (S^{n-1}_i\vee S^{n+1}_i))$ such that $F\phi=\sum_{i=1}^k [\iota_{n-1}^i,\iota_{n+1}^i]+\psi'$ with $\psi'\in B$.
\end{corollary}

Now we are ready to achieve our goal and prove Theorem \ref{thm:7}. For any $\phi\in\pi_{2n-1}(\bigvee_{i=1}^k (S^{n-1}_i\vee S^{n+1}_i))$, we define
\[
X(\phi)=\bigg(\bigvee_{i=1}^k (S^{n-1}_i\vee S^{n+1}_i)\bigg)\bigcup_\phi D^{2n}.
\]

\begin{lemma}\label{thm:14}
  There exists $\phi=\sum_{i=1}^k [\iota_{n-1}^i,\iota_{n+1}^i]+\psi$ with $\psi\in B$ such that $X\simeq X(\phi)$.
\end{lemma}

\begin{proof}
  By Corollary \ref{thm:13}, the lemma holds if we can prove that there exists $\phi=\sum_{i=1}^k [\iota_{n-1}^i,\iota_{n+1}^i]+\psi$ with $\psi\in B\oplus C$ such that $X\simeq X(\phi)$.

  Assume $X\simeq X(\phi)$ with $M_\phi=(a_{ij})$, ie,
  \[
  \phi=\sum_{1\leq i,j\leq k} a_{ij}[\iota_{n-1}^i,\iota_{n+1}^j]+\phi',
  \]
  where $\phi'\in B\oplus C$. Then $M_{\phi}$ defines the cup product in $X(\phi)$. Specifically, let $\alpha_i\in H^{n-1}(X(\phi))$, $\beta_j\in H^{n+1}(X(\phi))$, and $\gamma\in H^{2n}(X(\phi))$ represent the Kronecker dual of elements in $H_{n-1}(X(\phi))$, $H_{n+1}(X(\phi))$, and $H_{2n}(X(\phi))$ which are represented by $S^{n-1}_i$, $S^{n+1}_j$, and $D^{2n}$ respectively. Then $\alpha_i\cup\beta_j=a_{ij}\gamma$. Since $X(\phi)$ is a Poincar\'e complex, the Poincar\'e duality guarantees that $M_{\phi}$ is an invertible integral matrix. Let $I$ represent the identity matrix and choose $F=\begin{pmatrix}
                                                                      f & 0 \\
                                                                      0 & h
                                                                    \end{pmatrix}\in \mathcal{E}(\bigvee_{i=1}^k (S^{n-1}_i\vee S^{n+1}_i))$
with $Q_f=M_{\phi}^{-1}$ and $Q_h=I$. According to Lemma \ref{thm:11}, we can conclude that $M_{F\phi}=Q_fM_{\phi}Q_h^T=I$, or alternatively,
\[
F\phi=\sum_{i=1}^k [\iota_{n-1}^i,\iota_{n+1}^i]+\psi,
\]
where $\psi\in B\oplus C$. Lemma \ref{thm:12} ensures that $X(F\phi)\simeq X(\phi)$. Consequently, by substituting $\phi$ with $F\phi$, the lemma is established.
\end{proof}

\begin{proof}[Proof of Theorem \ref{thm:7}]
  By Lemma \ref{thm:14}, there exists
  \[
  \phi=\sum_{i=1}^k [\iota_{n-1}^i,\iota_{n+1}^i]+\sum_{i=1}^{k}\iota_{n-1}^i\alpha_i+\sum_{i=1}^{k}\iota_{n+1}^i\beta_i
  \]
  such that $X\simeq X(\phi)$, where $\alpha_i\in\pi_{2n-1}(S^{n-1})$ and $\beta_i\in\pi_{2n-1}(S^{n+1})$. Let $X_i=(S^{n-1}\vee S^{n+1})\cup_{\phi_i} D^{2n}$ where $\phi_i=[\iota_{n-1},\iota_{n+1}]+\alpha_i+\beta_i$. Then $H_{n-1}(X_i)\cong\mathbb{Z}$ and $X\simeq X_1\# X_2\#\cdots\# X_n$.
\end{proof}

\section{Proof of Theorem \ref{thm:15}}\label{sec:5}

Let $X$ be a $2$-connected $8$-dimensional Poincar\'e complex which has torsion-free homology and satisfies hypothesis (H) (where $n=4$). Suppose $\rank H_3(X)=k$. Define
\[
X_{r,s}=(S^3\vee S^5)\cup_{[\iota_3,\iota_5]+r\nu'\eta_6+s\eta_5^2} D^8
\]
for $r,s\in \{0,1\}$. In accordance with Theorem \ref{thm:7}, it can be inferred that $X\simeq X_1\# X_2\# \cdots \# X_k$, where each $X_i$ ($1\leq i\leq k$) belongs to $\{X_{0,0},X_{0,1},X_{1,0},{X_{1,1}}\}$. To prove Theorem \ref{thm:15}, all we need is to compare the homotopy types of connected sums derived from various combinations of $X_{r,s}$. We initially focus on the case where $k=2$.

\begin{proposition}\label{thm:17}
  \begin{enumerate}
    \item $X_{1,0}\#X_{1,0}\simeq X_{1,0}\# S^3\times S^5$;
    \item $X_{0,1}\# X_{0,1}\simeq X_{0,1}\# S^3\times S^5$;
    \item $X_{1,0}\# X_{0,1}\simeq X_{1,1}\# X_{1,1}$;
    \item $X_{1,1}\# X_{1,0}\simeq X_{1,1}\# X_{0,1}\simeq X_{1,1}\# S^3\times S^5$.
  \end{enumerate}
\end{proposition}

\begin{proof}
  Construct a self-map $F$ of $S_1^3\vee S_1^5\vee S_2^3\vee S_2^5$ as follows:
  \[
  F=\iota_3^1\vee (\iota_5^1-\iota_5^2)\vee (\iota_3^1+\iota_3^2)\vee \iota_5^2.
  \]
  It is straightforward to verify that $F\in \mathcal{E}(S_1^3\vee S_1^5\vee S_2^3\vee S_2^5)$. We analyze the action of $F$ on certain elements of $\pi_7(S_1^3\vee S_1^5\vee S_2^3\vee S_2^5)$:
  \begin{itemize}
    \item $F([\iota_3^1,\iota_5^1]+[\iota_3^2,\iota_5^2])= [\iota_3^1,\iota_5^1-\iota_5^2]+[\iota_3^1+\iota_3^2,\iota_5^2]= [\iota_3^1,\iota_5^1]+[\iota_3^2,\iota_5^2]$;
    \item $F(\iota_3^1\nu'\eta_6)=\iota_3^1\nu'\eta_6$;
    \item $F(\iota_5^1\eta_5^2)=(\iota_5^1-\iota_5^2)\eta_5^2= \iota_5^1\eta_5^2-\iota_5^2\eta_5^2=\iota_5^1\eta_5^2+\iota_5^2\eta_5^2$;
    \item $F(\iota_3^2\nu'\eta_6)=(\iota_3^1+\iota_3^2)\nu'\eta_6= \iota_3^1\nu'\eta_6+\iota_3^2\nu'\eta_6+[\iota_3^1,\iota_3^2]\eta_5^2$ (see Lemma \ref{thm:18}).
  \end{itemize}
  Combining the above, we have:

  (1) $F([\iota_3^1,\iota_5^1]+[\iota_3^2,\iota_5^2]+\iota_3^2\nu'\eta_6)= [\iota_3^1,\iota_5^1]+[\iota_3^2,\iota_5^2]+\iota_3^1\nu'\eta_6+\iota_3^2\nu'\eta_6+[\iota_3^1,\iota_3^2] \eta_5^2$. Combined with Lemmas \ref{thm:19} and \ref{thm:12}, this implies that $S^3\times S^5\# X_{1,0}\simeq X_{1,0}\# X_{1,0}$.

  (2) $F([\iota_3^1,\iota_5^1]+[\iota_3^2,\iota_5^2]+\iota_5^1\eta_5^2)= [\iota_3^1,\iota_5^1]+[\iota_3^2,\iota_5^2]+\iota_5^1\eta_5^2+\iota_5^2\eta_5^2$, which means $X_{0,1}\# S^3\times S^5\simeq X_{0,1}\# X_{0,1}$.

  (3) $F([\iota_3^1,\iota_5^1]+[\iota_3^2,\iota_5^2]+\iota_3^2\nu'\eta_6+\iota_5^1\eta_5^2)= [\iota_3^1,\iota_5^1]+[\iota_3^2,\iota_5^2]+\iota_3^1\nu'\eta_6+\iota_3^2\nu'\eta_6+ \iota_5^1\eta_5^2+\iota_5^2\eta_5^2 +[\iota_3^1,\iota_3^2]\eta_5^2$, which indicates that $X_{0,1}\# X_{1,0}\simeq X_{1,1}\# X_{1,1}$.

  (4) $F([\iota_3^1,\iota_5^1]+[\iota_3^2,\iota_5^2]+\iota_3^1\nu'\eta_6+\iota_5^1\eta_5^2)= [\iota_3^1,\iota_5^1]+[\iota_3^2,\iota_5^2]+\iota_3^1\nu'\eta_6+\iota_5^1\eta_5^2+\iota_5^2\eta_5^2$, which suggests that $X_{1,1}\# S^3\times S^5\simeq X_{1,1}\# X_{0,1}$.

        $F([\iota_3^1,\iota_5^1]+[\iota_3^2,\iota_5^2]+\iota_3^2\nu'\eta_6+\iota_5^2\eta_5^2)= [\iota_3^1,\iota_5^1]+[\iota_3^2,\iota_5^2]+\iota_3^1\nu'\eta_6+\iota_3^2\nu'\eta_6+ \iota_5^2\eta_5^2 +[\iota_3^1,\iota_3^2]\eta_5^2$, which means that $S^3\times S^5\# X_{1,1}\simeq X_{1,0}\# X_{1,1}$.
\end{proof}

Before moving to the general case, we give the following lemma.

\begin{lemma}\label{thm:21}
  For $\alpha\in\pi_5(\bigvee_{i=1}^k S^3_i)\subset \pi_5(\bigvee_{i=1}^k (S^3_i\vee S^5_i))$, we have $[\iota_3^i,\alpha]\in C$ where $1\leq i\leq k$.
\end{lemma}

\begin{proof}
  By the Hilton-Milnor theorem,
  \[
  \pi_5(\bigvee_{i=1}^k S^3_i)\cong  \bigg(\bigoplus_{i=1}^k\mathbb{Z}_2\{\iota_3^i\eta_3^2\}\bigg)\oplus\bigg(\bigoplus_{1\leq i<j\leq k}\mathbb{Z}\{[\iota_3^i,\iota_3^j]\}\bigg).
  \]
  Suppose $\alpha=\sum_{j=1}^{k}a_j\iota_3^j\eta_3^2+\sum_{1\leq s<t\leq k} b_{st} [\iota_3^s,\iota_3^t]$. Then for $1\leq i\leq k$,
  \begin{align*}
    [\iota_3^i,\alpha] & =\bigg[\iota_3^i,\sum_{j=1}^{k}a_j\iota_3^j\eta_3^2+\sum_{1\leq s<t\leq k} b_{st} [\iota_3^s,\iota_3^t]\bigg] \\
     & =\sum_{j=1}^{k}a_j[\iota_3^i,\iota_3^j\eta_3^2]+\sum_{1\leq s<t\leq k} b_{st}[\iota_3^i,[\iota_3^s,\iota_3^t]].
  \end{align*}
  Together with Lemma \ref{thm:4}, the fact $[\iota_3,\iota_3]=0$, and the Jacob identity for Whitehead product, one can successfully complete the proof.
\end{proof}

\begin{proof}[Proof of Theorem \ref{thm:15}]
  Given that the case $k=1$ has already been solved in \cite{Wan21}, we proceed under the assumption that $k\geq 2$. By applying Proposition \ref{thm:17}, one can directly verify that $X$ must be homotopy  equivalent to one of the complexes presented in (1)-(5).

  To prove that these complexes are not homotopy equivalent, we still use Lemma \ref{thm:12} as our main tool. The attaching maps of the top cell for complexes (1)-(5) are presented below:
  \begin{gather*}
    \phi_1=\sum_{i=1}^k [\iota_3^i,\iota_5^i],\quad \phi_2=\phi_1+\iota_3^1\nu'\eta_6,\quad \phi_3=\phi_1+\iota_5^1\eta_5^2, \\
    \phi_4=\phi_1+\iota_3^1\nu'\eta_6+\iota_5^1\eta_5^2,\quad \phi_5=\phi_1+\iota_3^1\nu'\eta_6+\iota_5^2\eta_5^2.
  \end{gather*}
  Let
  $F=\begin{pmatrix}
     f & g \\
     0 & h
   \end{pmatrix}\in\mathcal{E}(\bigvee_{i=1}^k (S^3_i\vee S^5_i))$, where $
   f\in \mathcal{E}(\bigvee_{i=1}^k S^3_i)$, $g\in [\bigvee_{i=1}^k S^5_i,\bigvee_{i=1}^k S^3_i]$ and $h\in \mathcal{E}(\bigvee_{i=1}^k S^5_i)$. All we need is to show that $F\phi_i\neq \phi_j$ for all $1\leq i\neq j\leq 5$.

   Suppose $Q_f=(a_{ij}), Q_h=(b_{ij})\in GL_k(\mathbb{Z})$ and set $g_i:=g|_{S_i^5}\in\pi_5 (\bigvee_{i=1}^k S_i^3)$. By definition, $M_{\phi_i}=I$ for any $1\leq i\leq 5$. Therefore, by Lemma \ref{thm:11}, $M_{F\phi_i}=Q_fQ_h^T$, and if $F\phi_i=\pm\phi_j$ for some $1\leq i,j\leq 5$, it must be true that $Q_fQ_h^T=\pm I$. Given this condition, we have
   \[
   F\phi_1=\pm\phi_1+\sum_{i=1}^k [\iota_3^i,g_i].
   \]
   According to Lemma \ref{thm:21}, $\sum_{i=1}^k [\iota_3^i,g_i]\in C$. This implies that $F\phi_1\neq \pm\phi_i$ for any $2\leq i\leq 5$.

   Using Lemma \ref{thm:22} and Corollary \ref{thm:23}, we have
   \begin{multline*}
     F(\iota_3^1\nu'\eta_6)=\bigg(\sum_{i=1}^{k} a_{i1}\iota_3^i\bigg)\nu'\eta_6=\sum_{i=1}^{k}(a_{i1}\iota_3^i)\nu'\eta_6+\sum_{1\leq i<j\leq k}[a_{i1}\iota_3^i,a_{j1}\iota_3^j]\eta_5^2 \\
     =\sum_{i=1}^{k}a_{i1}\iota_3^i\nu'\eta_6+\sum_{1\leq i<j\leq k}a_{i1}a_{j1}[\iota_3^i,\iota_3^j]\eta_5^2.
   \end{multline*}
   Therefore, $F\phi_2=F(\phi_1)+F(\iota_3^1\nu'\eta_6)$ does not contain any term of $\iota_5^i\eta_5^2$, leading to the conclusion that $F\phi_2\neq \pm \phi_3,\pm\phi_4,\pm\phi_5$.

   Similarly, since
   \[
   F(\iota_5^1\eta_5\eta_6)=\bigg(\sum_{i=1}^{k} b_{i1}\iota_5^i\bigg)\eta_5\eta_6=\sum_{i=1}^{k} b_{i1}\iota_5^i\eta_5^2,
   \]
   we see $F\phi_3=F(\phi_1)+F(\iota_5^1\eta_5^2)$ does not contain any term of $\iota_3^i\nu'\eta_6$, suggesting that $F\phi_3\neq \pm \phi_2,\pm\phi_4,\pm\phi_5$.

   It remains to demonstrate that $F\phi_4\neq \pm\phi_5$. As previously discussed, under the assumption $Q_fQ_h^T=\pm I$, we have
   \[
   F\phi_4=F\phi_1+F(\iota_3^1\nu'\eta_6)+F(\iota_5^1\eta_5^2) =\pm\phi_1+\sum_{i=1}^{k}a_{i1}\iota_3^i\nu'\eta_6+\sum_{i=1}^{k} b_{i1}\iota_5^i\eta_5^2+\psi,
   \]
   where $\psi\in C$. If $F\phi_4=\pm\phi_5=\pm\phi_1+\iota_3^1\nu'\eta_6+\iota_5^2\eta_5^2$, then for $1\leq i\leq k$, when modulo $2$,
   \[
   a_{i1}\equiv\begin{cases}
            1, & i=1 \\
            0, & i\neq 1,
          \end{cases}
   \quad
   b_{i1}\equiv\begin{cases}
            1, & i=2 \\
            0, & i\neq 2.
          \end{cases}
   \]
   Therefore, $\sum_{i=1}^{k}a_{i1}b_{i1}\equiv 0 \mod 2$. However, the condition $Q_fQ_h^T=\pm I$ implies that $\sum_{i=1}^{k}a_{i1}b_{i1}=\pm 1$, thereby creating a contradiction. Hence it can be concluded that $F\phi_4\neq \pm\phi_5$, and our proof is now fully complete.
\end{proof}

\section*{Acknowledgments}

This work is supported by the National Natural Science Foundation of China (Grant
No. 12101019).

\bibliographystyle{plain}
\bibliography{bib.bib}

\noindent Xueqi Wang

\noindent Department of Basic Sciences, Beijing International Studies University, Bejing, 100024, China

\noindent wangxueqi@bisu.edu.cn

\end{document}